\documentclass[11pt,a4paper,fleqn]{scrartcl}    
\usepackage[english]{babel}                     
\usepackage[T1]{fontenc}                   
\usepackage[applemac]{inputenc}
\usepackage{graphicx}                       % 
\usepackage[font=small]{quoting}      
\usepackage{enumitem,wrapfig}
\usepackage[colorlinks=true,linkcolor=blue]{hyperref}
\usepackage{xcolor}
\usepackage{pgf,tikz}
\usepackage{picture}
\usepackage{amsmath,amsthm,amsfonts,amssymb}
\newtheorem{theorem}{Theorem}[section]
\newtheorem{lemma}{Lemma}[section]
\newtheorem{prop}{Proposition}[section]
\newtheorem{cor}{Corollary}[section]
\newtheorem{op}{Open Problem}

\theoremstyle{definition}

\newtheorem{rem}{Remark}[section]

\newcommand{\ds}{\displaystyle}

\newcommand{\R}{\mathbb R}

\newcommand{\de}{\partial}
\newcommand{\eps}{\varepsilon}

\DeclareMathOperator{\dist}{dist}

\begin{document}
\title{An optimization problem in thermal insulation with Robin boundary conditions}
 \author{Francesco Della Pietra$^{*}$ \\ Carlo Nitsch$^{*}$\\ Riccardo Scala$^{**}$\\ Cristina Trombetti%
       \thanks{Universit\`a degli studi di Napoli Federico II, Dipartimento di Matematica e Applicazioni ``R. Caccioppoli'', Via Cintia, Monte S. Angelo - 80126 Napoli, Italia.
       Email: f.dellapietra@unina.it, carlo.nitsch@unina.it, cristina.trombetti@unina.it  \newline $^{**}$ Universit\`a degli studi di Siena, Dipartimento di Ingegneria dell'Informazione e Scienze Ma\-te\-ma\-ti\-che, San Niccol\`o, via Roma, 56, 53100 Siena - Italia, riccardo.scala@unisi.it }}
\maketitle
\begin{abstract}
We study thermal insulating of a bounded body $\Omega\subset \R^n$. Under a prescribed heat source $f\geq0$,  we consider a model of heat transfer between $\Omega$ and the environment determined by convection; this corresponds, before insulation, to Robin boundary conditions. The body is then surrounded by a layer of insulating material of thickness of size $\eps>0$, and whose conductivity is also proportional to $\eps$. This corresponds to the case of a small amount of insulating material, with excellent insulating properties. We then compute the  $\Gamma$-limit of the energy functional $F_\eps$ and prove that this is a functional $F$ whose minimizers still satisfy an elliptic PDEs system with a non uniform Robin boundary condition depending on the distribution of insulating layer around $\Omega$. In a second step we study the maximization of heat content (which measures the goodness of the insulation) among all the possible distributions of insulating material with fixed mass, and prove an optimal upper bound in terms of geometric properties. Eventually we prove a conjecture in \cite{bbn2} which states that the ball surrounded by a uniform distribution of insulating material maximizes the heat content.
\end{abstract}



\section{Introduction}

Insulation is an old and multidisciplinary subject which is gaining more and more popularity after \emph{efficient energy use} has become a top environmental priority. While physicists and engineers are mainly concerned with new materials and technologies, a hot topic in mathematics is shape optimization. The question of thermal insulation of a body, when the heat transfer with environment is conveyed by conduction (Dirichlet boundary conditions) has been considered in the contribution \cite{ab80}, where the generality of the problem of reinforcement covers the setting of a heated material body, which we want to thermally insulate with the aid of a small amount of material surrounding it. This material is distributed on the boundary of our reference configuration $\Omega$, under the form of a thin insulating layer. Before \cite{ab80}, the reinforcement problem has been faced in \cite{brca80,cf80}, 
where related results have been obtained. The general setting considered in these papers (see also \cite{cakri,f80,d99} for related analyses) is the following: if $\Omega$ represents a body and $f$ a given heat source, one has to minimize the functional 
\begin{align}
 G_\eps(u):=\frac12\int_\Omega |\nabla u|^2dx+\frac\eps2\int_{\Sigma_\eps}|\nabla u|^2dx-\int_\Omega f u \;dx,
\end{align}
among all $u\in H^1_0(\bar \Omega\cup \Sigma_\eps)$. Here $\Sigma_\eps$ is  the volume, around $\Omega$, occupied by a thermal insulating material, of the form
\begin{align}\label{def_Sigmaeps}
 \Sigma_\eps:=\{x:x=\sigma+\eps sh(\sigma)\nu(\sigma),\;\sigma\in \partial \Omega,\;s\in(0,1)\},
\end{align}
with $\eps>0$, $\nu(\sigma)$ the outer normal to $\partial \Omega$ at $\sigma$, $h:\partial\Omega\mapsto (0,+\infty)$ a positive function, with suitably regularity, which accounts for the thickness of the layer at any point $\sigma\in \partial\Omega$.
The functional $G_\eps$ is defined on the space $H^1_0$, which means that by convention we are assuming the external temperature is equal to $0$. 
A minimizer $u_\eps$ of $G_\eps$ solves the system of partial differential equations in $\Omega_{\eps}=\bar \Omega\cup \Sigma_{\eps}$,
\[
\begin{cases}
-\Delta u_{\eps}=f&\hbox{in }\Omega\\[.2cm]
-\Delta u_{\eps}=0&\hbox{in }\Sigma_\eps\\[.2cm]
\ds u_{\eps}=0&\hbox{on }\partial\Omega_\eps\\[.2cm]
\ds\frac{\partial u_{\eps}^-}{\partial\nu}
=\eps\frac{\partial u_{\eps}^+}{\partial\nu}&\hbox{on }\partial\Omega.
\end{cases}
\]
Here  $\ds\frac{\partial u_{\eps}^-}{\partial\nu}$ and $\ds\frac{\partial u_{\eps}^+}{\partial\nu}$ represent the normal derivatives of $u_\eps$ from inside and outside $\Omega$, respectively. Indeed, the last equation is a heat transmission condition between the body $\Omega$ and the layer $\Sigma_\eps$, whose conductivity is much smaller than the one of $\Omega$, and it is assumed to have value $\eps$. Hence, both the conductivity and the thickness of the insulating layer are assumed to be of order $\eps$.

The analysis of the behaviour of $u_\eps$ as $\eps\rightarrow0$ has been accounted in \cite{ab80} via a $\Gamma$-convergence technique, which leads to the limiting functional 
\begin{align}
 G(u)=\frac12\int_\Omega |\nabla u|^2dx+\int_{\partial\Omega}\frac {u^2}{h}d\mathcal H^{n-1}(\sigma)-\int_\Omega f u\; dx,
\end{align}
defined on $H^1(\Omega)$.
A minimizer of $G$ solves 
\begin{equation}\label{bubuni}
\begin{cases}
-\Delta u=f&\hbox{in }\Omega\\[.2cm]
h\frac{\de u }{\de \nu}+u =0&\hbox{on }\partial\Omega.
\end{cases}
\end{equation}
Notice that here the function $h$ is fixed. The problem of determining the optimal distribution $h$, given a fixed amount $m$ of insulating material
$$m=\int_{\partial\Omega}h \;d\mathcal H^{n-1}(\sigma),$$
has been studied in \cite{bbn2} (see also \cite{bbn1}) together with the analysis of the optimal shape of $\Omega$ in order to maximize the heat insulation. For $\Omega$ and $h$ smooth enough and $f\equiv 1$ a solution $u$ to \eqref{bubuni} is such that the functional
$$G(u)= -\frac{1}{2}\int_\Omega u$$
is proportional (but with opposite sign) to the heat content. Therefore the smallest the minima of $G$ the better the insulation.
The optimal $h$ for a fixed domain $\Omega$, leads indeed to minimize the functional $G$ in both $u$ and $h$, and the Authors observed that for minima, $u$ and $h$ must be proportional on $\partial\Omega$ (the ratio will be a constant depending on the total mass $m$). To this concern, in \cite{bbn2} the following open problem has been raised:
\begin{op}\label{op1} Assuming constant heat source $f\equiv 1$. For a given amount of insulating material $m$ and given volume of the conductor $|\Omega|$, find the shape of $\Omega$ (and the related optimal distribution of insulating material $h$), which maximizes the heat content, strictly speaking the shape which maximizes $\displaystyle \int_\Omega u \,dx$, when $u$ solves \eqref{bubuni}.
\end{op}  %an interesting feature of symmetry breaking was observed in \cite{bbn2}, where it is shown that for small amount of insulating material ($m<m_0$) the ball is best insultated for non-uniform distribution of $h$ on its perimeter. DIRE QUI DEL PROBLEMA AGLI AUTOVALORI
Later in the paper, we shall prove that balls and only balls, with an insulation layer uniformly distributed around the boundary, provide the answer to such a problem. But first let us describe the main result of our paper.

%In the present paper w
Our main goal is to study the $\Gamma$-convergence of functionals similar to $G_\eps$ which account for Robin boundary condition instead of Dirichlet one. The interest in Robin boundary conditions is justified by the fact that they describe convection heat transfer, which is a major mode of heat tranfer. Think for example the way a boiler, or a cup of tea, or a building, exchange heat with environment across their surface. 
This leads to the new functionals $F_\eps$, defined on $H^1(\R^n)$, given by
\begin{align}
F_{\eps}(u)=\frac12\int_\Omega|\nabla u|^2\,dx+\frac\eps2\int_{\Sigma_\eps}|\nabla u|^2\,dx+\frac{\beta}{2}\int_{\de \Omega_{\eps}} u^{2}d\mathcal H^{n-1}(\sigma) -\int_\Omega fu\,dx,
\end{align}
where $\beta>0$ is a fixed parameter. Minimizers of $F_\eps$ are solutions to
\begin{align}
\begin{cases}
-\Delta u_{\eps}=f&\hbox{in }\Omega\\[.2cm]
-\Delta u_{\eps}=0&\hbox{in }\Sigma_\eps\\[.2cm]
\ds\frac{\de u_{\eps}}{\de \nu}+\beta u_{\eps}=0&\hbox{on }\partial\Omega_\eps\\[.2cm]
\ds\frac{\partial u_{\eps}^-}{\partial\nu}
=\eps\frac{\partial u_{\eps}^+}{\partial\nu}&\hbox{on }\partial\Omega.
\end{cases} 
\end{align}
The b.c. $\ds\frac{\de u_{\eps}}{\de \nu}+\beta u_{\eps}=0$ tell us that the out-flux of heat is proportional to the temperature jump across the body surface, and once again this reminds us that the external temperature is set conventionally to zero. Once again both the conductivity and the thickness of the insulating layer are of order $\eps$.

We are interested in the asymptotics of $F_\eps$ as $\eps\rightarrow0$ and indeed we prove that, under suitable hypotheses of the regularity of $\partial\Omega$ and $h$, the functionals $F_\eps$ $\Gamma$-converge to
\begin{align}
F(u)=\frac{1}{2}\int_{\Omega} |\nabla u|^{2} dx +\frac{\beta}{2}\int_{\de\Omega}\frac{u^{2}}{1+\beta h}d\mathcal H^{n-1}(\sigma) -\int_{\Omega} fu\,dx,
\end{align}
with respect to the $L^2(\R^n)$ topology.

In a second part of the paper we deal with the problem of finding the optimal distribution of insulating material with a fixed mass $m$. 
In contrast to the homogeneous Dirichlet boundary condition case, the minimization with respect to $h$ is, in our case, more involved. For a given $u$, in order to prove existence of a minimizer $h\in L^2(\partial\Omega)$ we have to show some compactness estimates which would avoid concentration phenomena. It turns out that for optimality to hold, it is necessary that the ratio between $u$ and $1+\beta h$ is constant, at least on a part of the boundary $\partial\Omega$, that coincides with the set where the trace of $u$ exceeds a given threshold, which depends on $m$. The exact value of this threshold, denoted by $c_u$, is not explicit but depends implicitly on $h$ itself (see formula \eqref{c_u} below).

Finally, we prove a sharp upper bound for the heat content $\int_\Omega u\,dx$. Maximizing such a quantity among all domains $\Omega$ of given volume, is clearly the ultimate goal of the shape optimization. Our estimate provides an answer to Open Problem 1.

{%\color{blue}
The paper is organized as follows: In section 
\ref{sec:setting} we set the notation,  describe our functional and the $\Gamma$-convergence analysis we will deal with. In section \ref{sec:3} we state and prove our main result concerning the computation of the $\Gamma$-limit $F$. In section \ref{sec:4} we pass to the analysis of the functional $F$; we show the existence and uniqueness of minimizers and we also characterize the shape of the optimal insulating layer $h$ for any given admissible variable $u$. Finally in Section \ref{sec:5} we deal with the shape optimization problem of maximizing the heat content among all the domains $\Omega$ and insulating distributions.  
}

\section{Setting of the problem}\label{sec:setting}
Let $\Omega$ be a bounded open set in $\R^{n}$, $n\ge 2$, with boundary of class $C^{1,1}$, and let $\nu$ denote its outer normal.
Let us fix a small parameter $\eps>0$, and let $h\colon \de\Omega\to \R$ be a bounded positive Lipschitz-continuous function.
The set $\Omega$
is surrounded by the layer $\Sigma_{\eps}$ defined in \eqref{def_Sigmaeps}, that we re-parametrize as follows
\[
\Sigma_{\eps}=\{\sigma +t\nu(\sigma),\;\sigma\text{ on }\de\Omega,\; 0<t<\eps h(\sigma)\}.
\]
We denote by
\[
\Omega_{\eps}=\bar \Omega\cup \Sigma_{\eps}.
\]
The problem of thermal insulation of $\Omega$ is governed by the following functional
\begin{align}\label{def_Feps}
F_{\eps}(v,h)=\frac12\int_\Omega|\nabla v|^2\,dx+\frac\eps2\int_{\Sigma_\eps}|\nabla v|^2\,dx+\frac{\beta}{2}\int_{\de \Omega_{\eps}} v^{2}d\mathcal H^{n-1}(\sigma) -\int_\Omega fv\,dx,
 \end{align}
where $f\in L^2(\Omega)$ is a non-negative function representing the external heating of $\Omega$,  and $\beta>0$ is a fixed parameter.
For $h$ fixed, let us consider the following minimization problem:
\[
\min_{v\in H^{1}(\Omega_{\eps})} F_{\eps}(v,h).
\]
The minimum $u_\eps$ will satisfy the following system of partial differential equations
\[
\begin{cases}
-\Delta u_{\eps}=f&\hbox{in }\Omega\\[.2cm]
-\Delta u_{\eps}=0&\hbox{in }\Sigma_\eps\\[.2cm]
\ds\frac{\de u_{\eps}}{\de \nu}+\beta u_{\eps}=0&\hbox{on }\partial\Omega_\eps\\[.2cm]
\ds\frac{\partial u^-_\eps}{\partial\nu}
=\eps\frac{\partial u^+_\eps}{\partial\nu}&\hbox{on }\partial\Omega.
\end{cases}
\]
Our asymptotic analysis concerns the behaviour of the functional $F_\eps(\cdot,h)$ as $\eps\rightarrow0$. For convenience we extend $F_\eps$ to $L^2(\Omega_\eps)$ setting $F_\eps(v,h)=+\infty$ if $v\notin H^1(\Omega_\eps)$.
We will prove that the $\Gamma-$limit of $F_{\eps}$ with respect to the $L^2(\R^n)$ topology is the functional $F$  defined by
\begin{equation}\label{def_F}
F(v,h)=\frac{1}{2}\int_{\Omega} |\nabla v|^{2} dx +\frac{1}{2}\int_{\de\Omega}\frac{v^{2}}{1+\beta h}d\mathcal H^{n-1}(\sigma) -\int_{\Omega} fv\,dx. 
\end{equation}
A minimizer of the limit functional $F$, denoted by $u\in H^{1}(\Omega)$, i.e. 
\[
F(u,h)=\min_{v\in H^{1}(\Omega)} F(v,h),
\]
satisfies the system
\begin{equation}\label{limiting_pb}
\begin{cases}
-\Delta u=f&\hbox{in }\Omega\\[.2cm]
(1+\beta h)\ds\frac{\de u }{\de \nu}+\beta u =0&\hbox{on }\partial\Omega.
\end{cases}
\end{equation}
{%\color{blue}
Notice that the boundary condition for such a minimizer is still of Robin type, with coefficients depending also on the value of $h$.
}

\section{$\Gamma$-convergence}\label{sec:3}

In this section we study the limit of the functionals $F_\eps$ as $\eps\rightarrow0$. 
{%\color{blue} 
Since, in this section, $h$ is a fixed positive and bounded function, we drop the dependence on $h$ in the functionals $F_\eps$ and $F$ defined in \eqref{def_Feps} and \eqref{def_F}, writing $F_\eps(\cdot)=F_\eps(\cdot,h)$ (and similarly for $F$). }
We first extend $F_\eps$ on $L^2(\R^n)$ as follows
\begin{align}
\label{gamma1}
 \bar F_{\eps}(v)=\begin{cases}
                F_{\eps}(v)&\text{if }v\in H^1(\Omega_\eps),\\
                +\infty&\text{otherwise in }L^2(\R^n)\setminus H^1(\Omega_\eps).
                \end{cases}
\end{align}
The main result of this section is the following.
\begin{theorem}
Let $\Omega$ be a bounded open set in $\R^{n}$, $n\ge 2$, with $C^{1,1}$ boundary. Let $h$ be a {%\color{blue} 
fixed positive Lipschitz function}. 
Let $\bar F_{\eps}$ be as in \eqref{gamma1}. Then $\bar F_\eps$ $\Gamma$-converges to the limit functional 
\begin{align*}
%\label{gamma2}
 \bar F(v)=\begin{cases}
                F(v)&\text{if }v\in H^1(\Omega),\\
                +\infty&\text{otherwise in }L^2(\R^n)\setminus H^1(\Omega).
                \end{cases}
\end{align*}
with respect to the $L^2(\R^n)$ topology.
\end{theorem}
\begin{proof}
As usual, we split the proof into two main steps, namely providing a $\Gamma$-liminf and $\Gamma$-limsup inequalities. First, we show that for any $v\in L^2(\R^n)$ and any sequence $v_\eps\in L^2(\R^n)$ converging to $v$ in $L^2(\R^n)$, we have
\begin{equation}\label{liminf}
 \liminf _{\eps\rightarrow0}\bar F_\eps(v_\eps)\geq \bar F(v).
\end{equation}
Second, we will show that for any $v\in H^{1}(\Omega)$ there exists a sequence $v_{\eps}\in H^{1}(\R^{n})$ such that 
\begin{equation}
\label{limsup0}
\limsup_{\eps\to 0} \bar F_{\eps}(v_{\eps}) \le \bar F(v).
\end{equation}

\noindent\textit{Step 1: proof of the $\Gamma$-liminf inequality \eqref{liminf}.}
Let us assume without loss of generality that $v_\eps\in H^1(\R^n)$, that the liminf is a limit, and that it is finite. In particular there is a constant $C>0$ such that 
\begin{equation}\label{bound}
 \bar F_\eps(v_\eps)<C.
\end{equation}
From now on we denote by $C$ a generic positive constant which might change from line to line.
First we observe that since the $v_\eps\rightarrow v$ in $L^2(\R^n)$ we straightforwardly have
\begin{equation*}
\int_\Omega f v_\eps \;dx\rightarrow \int_\Omega f v \;dx. 
\end{equation*}
Moreover, we also get, from \eqref{bound},
\begin{equation*}
\int_\Omega|\nabla v_\eps|^2dx<C, 
\end{equation*}
for some positive constant $C$. Thus we infer, up to subsequences,
\begin{equation}\label{convergence}
 v_\eps\rightharpoonup v\text{ weakly in }H^1(\Omega).
\end{equation}
Therefore $v\in H^1(\Omega)$.
As a consequence we find out
\begin{equation}
\liminf_{\eps\rightarrow0} \frac12\int_\Omega|\nabla v_\eps|^2dx-\int_\Omega f v_\eps \;dx\geq\frac12\int_\Omega|\nabla v|^2dx-\int_\Omega f v \;dx,
\end{equation}
and then, in order to prove \eqref{liminf}, it suffices to show 
\begin{align}\label{claimliminf}
 \liminf_{\eps\rightarrow0} \Big(\frac\eps2\int_{\Sigma_\eps}|\nabla v_\eps|^2\,dx+\frac{\beta}{2}\int_{\de \Omega_{\eps}} v_\eps^{2}\;d\mathcal H^{n-1}(\sigma)\Big)\geq \frac{\beta}{2}\int_{\de\Omega}\frac{v^{2}}{1+\beta h}d\mathcal H^{n-1}(\sigma).
\end{align}
To this aim, we focus on the first term on the left-hand side, and following the argument used in \cite{ab80} we write it as
\begin{equation}
\label{integrazione}
 \frac{\eps}{2}\int_{\Sigma_\eps}|\nabla v_\eps|^2\,dx=\frac\eps2\int_{\de \Omega}\int_0^{\eps h(\sigma)}|\nabla v_\eps(\sigma+t\nu(\sigma))|^2(1+\eps R(\sigma,t))\,dt\;d\mathcal H^{n-1}(\sigma),
\end{equation}
where $R(\sigma,t)$ is a suitable remainder, which, thanks to the Lipschitz continuity of $\nu$ and $h$, is uniformly bounded by a constant, i.e. $|R(\sigma,t)|\leq C$. Therefore
\begin{align*}
 \frac{\eps}{2}\int_{\Sigma_\eps}|\nabla v_\eps|^2\,dx\geq \frac{\eps(1-\eps C)}{2}\int_{\de \Omega}\int_0^{\eps h(\sigma)}|\nabla v_\eps(\sigma+t\nu(\sigma))|^2\,dt\;d\mathcal H^{n-1}(\sigma).
\end{align*}
Now by H\"older inequality, for all $\sigma\in \partial \Omega$,
\begin{align*}
 \int_0^{\eps h(\sigma)}|\nabla v_\eps(\sigma+t\nu(\sigma))|^2\,dt&\geq \frac{1}{\eps h(\sigma)}\Big(\int_0^{\eps h(\sigma)}|\nabla v_\eps(\sigma+t\nu(\sigma))|\,dt\Big)^2\\
 &\geq\frac{1}{\eps h(\sigma)}\Big(\int_0^{\eps h(\sigma)}\nabla v_\eps(\sigma+t\nu(\sigma))\cdot \nu(\sigma)\,dt\Big)^2\\
 &=\frac{\big(v_\eps(\sigma+\eps h(\sigma)\nu(\sigma))-v_\eps(\sigma)\big)^2}{\eps h(\sigma)}.
\end{align*}
We have then obtained 
\begin{align}\label{use1}
 \frac{\eps}{2}\int_{\Sigma_\eps}|\nabla v_\eps|^2\,dx\geq \frac{\eps(1-\eps C)}{2}\int_{\de \Omega}\frac{\big(v_\eps(\sigma+\eps h(\sigma)\nu(\sigma))-v_\eps(\sigma)\big)^2}{\eps h(\sigma)}d\mathcal H^{n-1}(\sigma).
\end{align}
Let us now focus on the second term in the left-hand side of \eqref{claimliminf}.
Again by a change of variables we find that 
\begin{equation*}
 \frac{\beta}{2}\int_{\de \Omega_{\eps}} v_\eps(\sigma)^{2}d\sigma=\frac\beta2\int_{\de \Omega}v_\eps(\sigma+\eps h(\sigma)\nu(\sigma))^2(1+\eps R'_\eps(\sigma))\,d\sigma,
\end{equation*}
where $R_\eps'(\sigma)$ is bounded by a constant, uniformly on $\partial\Omega$, thanks to the hypothesis of Lipschitz continuity of $\nu$ and $h$. In particular there is a constant $C>0$ such that 
\begin{equation}\label{use2}
 \frac{\beta}{2}\int_{\de \Omega_{\eps}} v_\eps(\sigma)^{2}d\mathcal H^{n-1}(\sigma)\geq\frac\beta2(1-\eps C)\int_{\de \Omega}v_\eps(\sigma+\eps h(\sigma)\nu(\sigma))^2\,d\mathcal H^{n-1}(\sigma).
\end{equation}
Going back to \eqref{claimliminf}, with \eqref{use1} and \eqref{use2} at our disposal, we are left to prove that 
\begin{align}\label{claimliminf2}
\liminf_{\eps\rightarrow0} &\Big(\frac\eps2\int_{\de \Omega}\frac{\big(v_\eps(\sigma+\eps h(\sigma)\nu(\sigma))-v_\eps(\sigma)\big)^2}{\eps h(\sigma)}d\mathcal H^{n-1}(\sigma)+\frac\beta2\int_{\de \Omega}v_\eps(\sigma+\eps h(\sigma)\nu(\sigma))^2\,d\mathcal H^{n-1}(\sigma)\Big)\nonumber\\
&\geq \frac{\beta}{2}\int_{\de\Omega}\frac{v(\sigma)^{2}}{1+\beta h(\sigma)}d\mathcal H^{n-1}(\sigma). 
\end{align}
We have used here that the term 
\begin{equation*}
 R(\eps):=\eps C\frac\eps2\int_{\de \Omega}\frac{\big(v_\eps(\sigma+\eps h(\sigma)\nu(\sigma))-v_\eps(\sigma)\big)^2}{\eps h(\sigma)}d\mathcal H^{n-1}(\sigma)+\eps C\frac\beta2\int_{\de \Omega}v_\eps(\sigma+\eps h(\sigma)\nu(\sigma))^2\,d\mathcal H^{n-1}(\sigma),
\end{equation*}
is infinitesimal as $\eps\rightarrow0$, because by \eqref{use1} and \eqref{use2} and by assumption \eqref{bound}, it turns out that, for $\eps$ small enough, $R(\eps)/\eps$ is bounded. 

In order to prove \eqref{claimliminf2} we use Young's inequality and write, for any $\lambda>0$, 
\begin{align*}
 &\frac\eps2\int_{\de \Omega}\frac{\big(v_\eps(\sigma+\eps h(\sigma)\nu(\sigma))-v_\eps(\sigma)\big)^2}{\eps h(\sigma)}d\mathcal H^{n-1}(\sigma)+\frac\beta2\int_{\de \Omega}v_\eps(\sigma+\eps h(\sigma)\nu(\sigma))^2\,d\mathcal H^{n-1}(\sigma)\\
 &=\int_{\de \Omega}\frac{v_\eps(\sigma+\eps h(\sigma)\nu(\sigma))^2+v_\eps(\sigma)^2-2v_\eps(\sigma+\eps h(\sigma)\nu(\sigma))v_\eps(\sigma)}{2 h(\sigma)}d\mathcal H^{n-1}(\sigma)\\
 &\;\;+\frac\beta2\int_{\de \Omega}v_\eps(\sigma+\eps h(\sigma)\nu(\sigma))^2\,d\mathcal H^{n-1}(\sigma)\\
 &\geq\frac12\int_{\de\Omega}\big(\frac{1-\lambda+\beta h(\sigma)}{h(\sigma)}\big)v_\eps(\sigma+\eps h(\sigma)\nu(\sigma))^2+\big(\frac{1}{h(\sigma)}-\frac{1}{\lambda h(\sigma)}\big)v_\eps(\sigma)^2d\mathcal H^{n-1}(\sigma),
\end{align*}
and so setting $\lambda=\lambda(\sigma):=1+\beta h(\sigma)$ we finally get
\begin{align*}
 &\frac\eps2\int_{\de \Omega}\frac{\big(v_\eps(\sigma+\eps h(\sigma)\nu(\sigma))-v_\eps(\sigma)\big)^2}{\eps h(\sigma)}d\mathcal H^{n-1}(\sigma)+\frac\beta2\int_{\de \Omega}v_\eps(\sigma+\eps h(\sigma)\nu(\sigma))^2\,d\mathcal H^{n-1}(\sigma)\\
 &\geq\frac\beta2\int_{\de\Omega}\frac{v_\eps(\sigma)^2}{1+\beta h(\sigma)}d\mathcal H^{n-1}(\sigma).
\end{align*}
Now \eqref{claimliminf2} follows by lower semicontinuity, thanks to \eqref{convergence}. The proof of the $\Gamma$-liminf inequality is achieved.
\newline

\noindent\textit{Step 2: proof of the $\Gamma$-limsup inequality \eqref{limsup}.}
Let $v\in H^{1}(\Omega)$. By the regularity of $\de \Omega$, we may extend $v$ in $H^{1}(\R^{n})$. Moreover, we extend $h$ in $\Sigma_\eps$ by setting $h(x)=h(\sigma)$, for all $x$ of the form $x=\sigma +t\nu(\sigma)\in \Sigma_{\eps}$ {%\color{blue} 
for some $t\in(0,\eps h(\sigma))$}.
Let us consider { $\varphi_{\eps}:\Sigma_\eps\rightarrow [0,1]$} given by
\[
\varphi_{\eps}(x)=1-\frac{\beta d(x)}{\eps(1+\beta h(x))},
\]
where $d(x)=\dist(x,\Omega)$. { Notice that $\varphi_\eps=1$ on $\partial \Omega$ but it is always strictly positive on $\partial \Omega_\eps$. We further extend $\varphi_\eps$ to (not relabelled) $\varphi_\eps\in L^2(\R^n)$ by setting $\varphi_\eps=1$ in $\Omega$ and $\varphi_\eps=0$ on $\R^n\setminus \Omega_\eps$.} We define  $v_{\eps}=v\varphi_{\eps}$, and we trivially see  that $v_{\eps}\to v$ in $L^{2}(\Omega)$ (since $v_\eps=v$ in $\Omega$). { Moreover an easy check shows that $v_\eps\in H^1(\Omega_\eps)$.} We will prove that
\[
\limsup_{\eps\to 0} \bar F_{\eps}(v_{\eps}) \le \bar F(v).
\]
To this aim, it suffices to prove that
\begin{equation}
\label{limsup}
\limsup_{\eps\to 0}\left( \eps \int_{\Sigma_{\eps}} \left|\nabla v_{\eps}\right|^{2}dx+
\beta \int_{\de\Omega_{\eps}} \left| v_{\eps}\right|^{2}d\mathcal H^{n-1}(\sigma) \right) \le {\beta} \int_{\de\Omega} \frac{v^{2}}{1+\beta h} d\mathcal H^{n-1}(\sigma).
\end{equation}
As regards the first term in the left hand side of \eqref{limsup}, let $\lambda\in(0,1)$ be arbitrary (we will choose it later). It holds
\begin{equation}
\label{pass}
{\eps} \int_{\Sigma_{\eps}} \left|\nabla v \varphi_{\eps}+\nabla \varphi_{\eps} v\right|^{2}dx\le 
\frac{\eps}{\lambda} \int_{\Sigma_{\eps}} \left| \nabla \varphi_{\eps}v \right|^{2} \,dx + \frac{\eps}{(1-\lambda)} \int_{\Sigma_{\eps}} \left|\nabla v \varphi_{\eps}\right|^{2}dx.
\end{equation}
Furthermore, { $d$ denoting the distance from $\Omega$}, we have
\begin{multline*}
\left|\nabla \varphi_{\eps}\right|^{2} \le \frac{\beta^{2}}{\eps^{2}} \left|\frac{\nabla d}{(1+\beta h)(1-d)}\right|^{2} (1-d)+\frac{\beta^{4}}{\eps^{2}} \left|\frac{\nabla h}{(1+\beta h)^{2}}\right|^{2}d\\
\le 
\frac{\beta^{2}}{\eps^{2}}  \frac{1}{(1+\beta h)^{2}(1-d)}  +\frac{\beta^{4}C}{\eps } \left |\nabla h\right|^{2},
\end{multline*}
{%\color{blue}
where we have used that $d\leq \eps h\leq \eps C$ in $\Sigma_\eps$.}
Hence
\[
\frac{\eps}{\lambda} \int_{\Sigma_{\eps}} \left| \nabla \varphi_{\eps}v  \right|^{2}  \,dx
\le \frac{\beta^{2}}{\eps\lambda} \int_{\Sigma_{\eps}} \frac{v^{2}}{(1+\beta h)^{2}(1-d)}dx
+\frac{\beta^{4}C}{\lambda} \int_{\Sigma_{\eps}} \left|\nabla h\right|^{2} v^{2}dx.
\]
On the other hand, reasoning similarly as in \eqref{integrazione}, we get
\begin{multline*}
\int_{\Sigma_{\eps}} \frac{v^{2}}{(1+\beta h)^{2}(1-d)}dx
= \\=\eps \int_{\de \Omega} h(\sigma) d\mathcal H^{n-1}(\sigma) \int_{0}^{1} 
\frac{v(\sigma+\eps t h(\sigma)\nu(\sigma))^{2}}{(1+\beta h(\sigma))^{2}(1-d(\sigma+\eps t h(\sigma)\nu(\sigma)))}(1+\eps R(\sigma,t))dt
\end{multline*}
where $\left|R(\sigma,t)\right| \le C$. Then, being $h$ Lipschitz, we get
\[
\limsup_{\eps\rightarrow0} \frac{\eps}{\lambda} \int_{\Sigma_{\eps}} \left| \nabla \varphi_{\eps}v  \right|^{2}  \,dx
\le 
\frac{\beta^{2}}{\lambda} \int_{\de \Omega} \frac{h(\sigma) v(\sigma)^2}{(1+\beta h(\sigma))^{2}} d\mathcal H^{n-1}(\sigma).
\]
Last term in \eqref{pass} vanishes as $\eps\to 0$, then, by the arbitrariness of $\lambda$ it holds that
\begin{equation}
\label{primotermine}
\limsup_{\eps\to 0} \eps \int_{\Sigma_{\eps}} \left|\nabla v_{\eps}\right|^{2}dx
\le
{\beta^{2}} \int_{\de \Omega} \frac{h(\sigma) v(\sigma)^{2}}{(1+\beta h(\sigma))^{2}} d\sigma.
\end{equation}
As regards the second term in the left hand side of \eqref{limsup}, by a change of variables we get that 
\[
\int_{\de\Omega_{\eps}}  v_{\eps}(\sigma)^{2}d\mathcal H^{n-1}(\sigma)=
\int_{\de \Omega} v_{\eps}(\sigma +\eps h(\sigma)\nu(\sigma))^{2}(1+\eps R'(\sigma))d\mathcal H^{n-1}(\sigma).
\]
Then, passing to the limit we obtain
\begin{equation}
\label{secondotermine}
\lim_{\eps\to 0}\int_{\de\Omega_{\eps}}  v_{\eps}(\sigma)^{2}d\mathcal H^{n-1}(\sigma)=
\int_{\de \Omega} \frac{v(\sigma)^{2}}{(1+\beta h(\sigma))^{2}}d\mathcal H^{n-1}(\sigma).
\end{equation}
Finally, using \eqref{primotermine} and \eqref{secondotermine} we get \eqref{limsup}.
Hence, the proof of the theorem is complete.
\end{proof}

\section{Analysis of the minimum}\label{sec:4}
In this section we consider the optimization of the shape of the insulating layer surrounding $\Omega$. 

{%\color{blue} 
In the computation of the $\Gamma$-limit of the functional $\bar{F}_\eps$ we have chosen and fixed a priori $h$ which was Lipschitz continuous, bounded, and strictly positive. Once the $\Gamma$-limit has been obtained, we are now interested into varying $h$ in order to achieve the optimal insulation.}

To this aim we first fix the amount of insulating material we want to exploit. 
For any positive number $m$, we set
\[
\mathcal H_{m}(\de\Omega)=\{h\in L^{1}(\de\Omega),\;h\ge 0 \colon \int_{\de \Omega} h\,d\sigma=m\}.
\]

We will now analyse the behaviour of the limit functional $\bar F(v,h)$. Specifically, we fix $m>0$ and allow $h\in \mathcal H_m(\partial\Omega)$ to vary. To address the optimality of the insulating problem with respect to $h\in\mathcal H_m(\partial\Omega)$ we study the minimization problem
\begin{align}\label{minimum_couple}
 &\min_{(v,h)\in H^1(\Omega)\times \mathcal H_m(\partial\Omega)} F(v,h)  \nonumber\\
 &=\min_{(v,h)\in H^1(\Omega)\times \mathcal H_m(\partial\Omega)}\left\{\frac{1}{2}\int_{\Omega} |\nabla v|^{2} dx +\frac{1}{2}\int_{\de\Omega}\frac{v^{2}}{1+\beta h}d\mathcal H^{n-1}(\sigma) -\int_{\Omega} fv\,dx  \; \right\}.
\end{align}

{%\color{blue} 
Notice that here we allow $h$ to be $0$ on some subdomain of $\partial \Omega$, and we also drop the requirement of Lipschitz continuity. This procedure is justified by the fact that, in general, the infimum of the functional among the class of positive Lipschitz continuous maps $h$ is not achieved since minimizers are not strictly positive (and neither Lipschitz continuous).}

We are now in position to state the following existence result:

{\begin{theorem}\label{main_teo}
Given any $\beta,m>0$, there exists a couple $(u,h)\in H^{1}(\Omega)\times L^{2}(\de \Omega)$, with $h\in \mathcal H_{m}(\partial \Omega)$, which minimizes \eqref{minimum_couple}. Moreover, 
\begin{align}\label{constant_cu}
 h(\sigma):=\begin{cases}
              \frac{|u(\sigma)|}{c_{u}\beta}-\frac1\beta&\text{if }|u(\sigma)|\ge c_{u},\\
              0&\text{otherwise },
             \end{cases} 
\end{align}
where $c_u$ is the unique positive constant satisfying 
\begin{align}\label{c_u_def}
 c_u=\Big(\frac{1}{|\{|u|\geq c_u\}|+m\beta}\Big)\int_{\{|u|\geq c_u\}}|u(\sigma)|d\mathcal H^{n-1}(\sigma).
\end{align}
Furthermore the couple $(u,h)$  is a solution to \eqref{limiting_pb}, and is also unique if the domain $\Omega$ is connected. 
\end{theorem}}
{%\color{blue}
\begin{rem}
 The constant $c_u$ is determined implicitly by equation \eqref{c_u_def}. As we will see (Lemma \ref{lemma1} below), for every $u\in H^1(\Omega)$ the constant $c_u$ is well-defined and positive whenever $u$ is not identically null on $\partial \Omega$. Moreover it depends on the amount of insulating material we have at disposal, namely the constant $m>0$.
\end{rem}
}

\subsection{Proof of Theorem \ref{main_teo}}

In order to prove Theorem \ref{main_teo} we will show some preliminary results. 
We start by the following Lemma, that states the existence and uniqueness of the constant $c_u$ appearing in \eqref{constant_cu}.

\begin{lemma}\label{lemma1}Let $\beta>0$, $m>0$ be fixed.
Let $v\in L^2(\de \Omega)$, and $m>0$ be fixed. Then there is a unique constant $c_v\geq 0$ such that 
\begin{align}\label{c_u}
 c_v=\Big(\frac{1}{|\{|v|\geq c_v\}|+m\beta}\Big)\int_{\{|v|\geq c_v\}}|v(\sigma)|d\mathcal H^{n-1}(\sigma).
\end{align}
Moreover $c_v=0$ if and only if $v=0$ $\mathcal H^{n-1}$-a.e. on $\de\Omega$.
\end{lemma}
\begin{proof}
 Let us consider the two functions $g_1,g_2:[0,+\infty)\rightarrow[0,+\infty)$ defined by
 \begin{align*}
  &g_1:c\mapsto \int_{|v|\geq c}(|v|-c)d\mathcal H^{n-1}(\sigma),\\
  &g_2:c\mapsto m\beta c,
 \end{align*}
which happen to be continuous on $[0,+\infty)$. Moreover, $g_2$ is strictly increasing, $g_2(0)=0$, whereas $g_1$ is non-increasing and $g_1(0)=\int_{\de\Omega}|v|d\mathcal H^{n-1}(\sigma)\geq0$. Moreover $g_1(0)=0$ if and only if $v\equiv0$. Since $g_2(c)\rightarrow+\infty$ as $c\rightarrow+\infty$, it follows that there is a unique $c_v\geq0$ such that $g_1(c_v)=g_2(c_v)$, and $c_v=0$ if and only if $v=0$ a.e. on $\de\Omega$.
\end{proof}

\begin{prop}\label{h_exists} Let $\beta>0$, $m>0$ be fixed, let $v\in L^2(\de \Omega)$, and let $h\in L^2(\de\Omega)$ be the function defined by
 \begin{align}\label{h_min}
  h(\sigma):=\begin{cases}
              \frac{|v(\sigma)|}{c_v\beta}-\frac1\beta&\text{if }|v(\sigma)|\geq c_v,\\
              0&\text{otherwise},
             \end{cases}
 \end{align}
where $c_v$ is the constant given by Lemma \ref{lemma1}.
Then $h$ is the solution to the minimum problem
\begin{equation}\label{minimum}
 \min_{\hat h\in \mathcal H_m(\partial \Omega)}\frac\beta2\int_{\de\Omega}\frac{v^2}{1+\beta \hat h}d\mathcal H^{n-1}(\sigma).
\end{equation}
\end{prop}
{%\color{blue}
\begin{rem}
It is easy to see that the function $h$ defined in \eqref{h_min} belongs to $\mathcal H_m(\partial \Omega)$. 
\end{rem}}

\begin{proof}
Let $\hat h\in L^2(\de\Omega)\cap\mathcal H_m(\partial\Omega)$ be a competitor for the problem \eqref{minimum}. We aim to prove that 
\begin{align}\label{claim_ineq}
 \int_{\de\Omega}\frac{v^2}{1+\beta \hat h}d\mathcal H^{n-1}(\sigma)\geq\int_{\de\Omega}\frac{v^2}{1+\beta h}d\mathcal H^{n-1}(\sigma).
\end{align}
We set, for $t\in[0,1]$,
\begin{equation*}
 \psi(t):=\int_{\de\Omega}\frac{v^2}{1+\beta (\hat h+t(h-\hat h))}d\mathcal H^{n-1}(\sigma),
\end{equation*}
in such a way that \eqref{claim_ineq} is equivalent to proving that  $\psi(0)\geq \psi(1)$. To this aim it suffices to check that $\psi'(t)\leq0$ for any $t\in[0,1]$.
Denoting by $A$ and $B$ the subsets of $\de\Omega$ defined as
$A:=\{\sigma:\hat h(\sigma)>h(\sigma)\}$, and $B:=\{\sigma:\hat h(\sigma)<h(\sigma)\}$ we have, for $t\in[0,1]$,
\begin{align*}
 &1+\beta (\hat h+t(h-\hat h))>1+\beta h,\quad \text{on } A,\\
 &1+\beta (\hat h+t(h-\hat h))<1+\beta h,\quad \text{on } B,
\end{align*}
hence it turns out
\begin{align*}
 \psi'(t)&=-\beta\int_{\de\Omega}\frac{v^2(h-\hat h)}{[1+\beta (\hat h+t(h-\hat h))]^2}d\mathcal H^{n-1}(\sigma)\\
 &=-\beta\int_{A}\frac{v^2(h-\hat h)}{[1+\beta (\hat h+t(h-\hat h))]^2}d\mathcal H^{n-1}(\sigma)-\beta\int_{B}\frac{v^2(h-\hat h)}{[1+\beta (\hat h+t(h-\hat h))]^2}d\mathcal H^{n-1}(\sigma)\\
 &\leq -\beta\int_{A}\frac{v^2(h-\hat h)}{(1+\beta h)^2}d\mathcal H^{n-1}(\sigma)-\beta\int_{B}\frac{v^2(h-\hat h)}{(1+\beta h)^2}d\mathcal H^{n-1}(\sigma)\\
 &=-\beta\int_{\{|v|\geq c_v\}}c_v^2(h-\hat h)d\mathcal H^{n-1}(\sigma)-\beta\int_{\{|v|< c_v\}}v^2(h-\hat h)d\mathcal H^{n-1}(\sigma).
\end{align*}
In the last equality we have used the explicit expression of $h$ in \eqref{h_min}.
Exploiting also that $h=0$ on $\{|v|< c_v\}$, we infer that the last line in the preceding expression equals
\begin{align*}
 &-\beta\int_{\{|v|\geq c_v\}}c_v^2(h-\hat h)\;d\mathcal H^{n-1}(\sigma)+\beta\int_{\{|v|< c_v\}}v^2\hat h\;d\mathcal H^{n-1}(\sigma)\\
 &\leq -\beta\int_{\{|v|\geq c_v\}}c_v^2(h-\hat h)\;d\mathcal H^{n-1}(\sigma)+\beta\int_{\{|v|< c_v\}}c_v^2\hat h\;d\mathcal H^{n-1}(\sigma)\\
%  &=-\beta\int_{\{|u|\geq c_u\}}c_u^2(h-\hat h)\;d\sigma-\beta\int_{\{|u|< c_u\}}c_u^2(h-\hat h)\;d\sigma\\
&=-\beta c_v^2\int_{\{|v|\geq c_v\}}h\;d\mathcal H^{n-1}(\sigma)+\beta c_v^2\int_{\{|v|< c_v\}}\hat h\;d\mathcal H^{n-1}(\sigma)=0.
\end{align*}
{%\color{blue}
In the last equality we have used the explicit expression \eqref{h_min} and the fact that  both $h$ and $\hat h$ have the same mass $m$. 
}
Thus $\psi'(t)\leq0$ for all $t\in[0,1]$, and the thesis is achieved.

\end{proof}
% In the analysis of the $\Gamma$-limit we introduce dependence on $h$ and define, for any $u\in H^1(\Omega)$,  $h\in L^2(\de\Omega)\cap \mathcal H_m$,
% \begin{equation}
%  F(v,h)=\frac{1}{2}\int_{\Omega} |\nabla v|^{2} dx +\frac{1}{2}\int_{\de\Omega}\frac{v^{2}}{1+\beta h}d\sigma -\int_{\Omega} fv\,dx.
% \end{equation}
In order to minimize $F$ with respect to $v$ and $h$, we need the following coerciveness result.

\begin{prop}\label{coer}
Let  $\beta>0$ and $m>0$ be fixed.
 There are positive constants $C_1,C_2$ such that for any $v\in H^1(\Omega)$, $h\in L^2(\de\Omega)\cap \mathcal H_m(\partial \Omega)$,
 \begin{align}\label{coerciveness}
  F(v,h)\geq C_1\|v\|^2_{H^1(\Omega)}-C_2.
 \end{align}
\end{prop}
\begin{proof}
 The term $\int_{\de\Omega}\frac{v^2}{1+\beta h}d\mathcal H^{n-1}(\sigma)$ is obviously greater or equal to $\int_{\de\Omega}\frac{v^2}{1+\beta h_v}d\mathcal H^{n-1}(\sigma)$, where $h_v$ is the function given in \eqref{h_min}. Now, recalling \eqref{c_u}, we estimate
 \begin{align*}
  \int_{\de\Omega}\frac{v^2}{1+\beta h_v}d\mathcal H^{n-1}(\sigma)&=\int_{\{|v|<c_v\}}\frac{v^2}{1+\beta h_v}d\mathcal H^{n-1}(\sigma)+\int_{\{|v|\geq c_v\}}\frac{v^2}{1+\beta h_v}d\mathcal H^{n-1}(\sigma)\\
  &=\int_{\{|v|<c_v\}}v^2\;d\mathcal H^{n-1}(\sigma)+\int_{\{|v|\geq c_v\}}c_v|v|\;d\mathcal H^{n-1}(\sigma)\\
  &=\int_{\{|v|<c_v\}}v^2\;d\mathcal H^{n-1}(\sigma)+\frac{\big(\int_{\{|v|\geq c_v\}}|v|\;d\mathcal H^{n-1}(\sigma)\big)^2}{|\{|v|\geq c_v\}|+m\beta}\\
  &\geq\frac{\big(\int_{\{|v|< c_v\}}|v|\;d\mathcal H^{n-1}(\sigma)\big)^2}{|\partial\Omega|}+\frac{\big(\int_{\{|v|\geq c_v\}}|v|\;d\mathcal H^{n-1}(\sigma)\big)^2}{Per(\Omega)+m\beta}\\
  &\geq \frac{\big(\int_{\de\Omega}|v|\;d\mathcal H^{n-1}(\sigma)\big)^2}{Per(\Omega)+m\beta}.
 \end{align*}
Thanks to classical trace inequality we thus conclude
\begin{align*}
&\frac{1}{2}\int_{\Omega} |\nabla v|^{2} dx +\frac{1}{2}\int_{\de\Omega}\frac{v^{2}}{1+\beta h}d\mathcal H^{n-1}(\sigma)\geq \frac{1}{2}\int_{\Omega} |\nabla v|^{2} dx +C\big(\int_{\de\Omega}|v|\;d\mathcal H^{n-1}(\sigma)\big)^2\\
&\geq C\|v\|_{H^1(\Omega)}^2,
\end{align*}
and we then obtain the thesis by standard arguments to estimate the linear term in the energy.
\end{proof}

Concerning uniqueness, we will need the following Proposition.

\begin{prop}
\label{uniqueness_prop}
{ Assume that the  set $\Omega$ is connected and let $\beta>0$, $m>0$ be fixed.} Then the  functional $F(v,h)$ satisfies the following convexity condition in $ H^{1}(\Omega)\times \mathcal H_{m}(\de\Omega)$, namely
\[
\frac{1}{2} \left[ F(v_{1},h_{1}) +F(v_{2},h_{2}) \right] > F\left(\frac{v_{1}+v_{2}}{2},\frac{h_{1}+h_{2}}{2}\right)\quad \forall (v_{1},h_{1})\neq (v_{2},h_{2}).
\]
\end{prop}

\begin{proof}
Let $v_{1},v_{2}$ be two functions in $H^{1}(\Omega)$. Then by Cauchy-Schwarz inequality
\begin{equation}
\label{equality0}
\int_{\Omega}\left(\frac{\left|\nabla v_{1}\right|^{2}}{2}+\frac{\left|\nabla v_{2}\right|^{2}}{2}\right)dx \ge \int_{\Omega} \left|\frac{\nabla (v_{1}+v_{2})}{2}\right|^{2}dx,
\end{equation}
and the equality holds if and only if $v_{2}=v_{1}+k$ (here is where we use that $\Omega$ is connected).
Moreover, 
\begin{multline} 
\label{eqpass}
\frac{1}{2}  \left(\frac{v_{1}^{2}}{1+\beta h_{1}}+\frac{v_{2}^{2}}{1+\beta h_{2}}\right)=\\[.2cm]=
\frac{v_{1}^{2}+v_{2}^{2}}{4+2\beta(h_{1}+h_{2})} +\left[\left(\frac{1+\beta h_{2}}{1+\beta h_{1}}\right)\frac{v_{1}^{2}}{2}+\left(\frac{1+\beta h_{1}}{1+\beta h_{2}}\right)\frac{v_{2}^{2}}{2}\right]\frac{1}{2+\beta(h_{1}+h_{2})} 
\end{multline}
By Cauchy-Schwarz inequality, it holds that
\[
\left(\frac{1+\beta h_{2}}{1+\beta h_{1}}\right)\frac{v_{1}^{2}}{2}+\left(\frac{1+\beta h_{1}}{1+\beta h_{2}}\right)\frac{v_{2}^{2}}{2} \ge v_{1}v_{2},
\]
with equality if and only if 
\begin{equation}
\label{equalitypass}
(1+\beta h_{2}) v_{1}= v_{2}(1+\beta h_{1}).
\end{equation}

Hence, equality in both \eqref{equality0} and \eqref{equalitypass} can hold if and only if $v_{2}=v_{1}$ (and $h_{2}=h_{1}$): indeed, equality in \eqref{equality0} gives $v_{2}=v_{1}+ k$ with $ k\ge 0$, and if \eqref{equalitypass} holds with $k>0$, then
\[
\beta h_{1}= (1+\beta h_{2})\frac{v_{1}}{v_{2}}-1 < \beta h_{2};
\]
so integrating on $\de\Omega$ we immediately get a contradiction. Then from \eqref{eqpass} we obtain that for $v_{2}=v_{1}+k$, with $k>0$,
\[
\frac{1}{2}  \left(\frac{v_{1}^{2}}{1+\beta h_{1}}+\frac{v_{2}^{2}}{1+\beta h_{2}}\right) > \frac{(\frac{v_{1}+v_{2}}{2})^{2}}{1+\beta \frac{h_{1}+h_{2}}{2}}.
\]
\end{proof}

%{\color{blue}
%\begin{rem}
% With small changes of the proof, it is possible to prove that the functional $F$ is strictly convex in $ H^{1}(\Omega)\times \mathcal H_{m}(\de\Omega)$, that is, for all $\lambda\in(0,1)$,
%\[
%\lambda F(v_{1},h_{1}) +(1-\lambda)F(v_{2},h_{2})  > F\left({\lambda v_{1}+(1-\lambda)v_{2}},{\lambda h_{1}+(1-\lambda)h_{2}}\right)\quad \forall (v_{1},h_{1})\neq (v_{2},h_{2}).
%\]
%\end{rem}
%However to show uniqueness of minimizers the case $\lambda=\frac12$ is sufficient. To simplify notation we have considered only this case.
%}
%\newline

We are now in position to conclude the proof of Theorem \ref{main_teo}.

\begin{proof}[Proof of Theorem \ref{main_teo}]
Let $(u_n,h_n)\in H^1(\Omega)\times \mathcal H_m(\partial \Omega)$ be a minimizing sequence; we can always assume $h_n$ smooth, and $u_n$ be the solution to
\begin{align}\label{PDEu_n}
 \begin{cases}
-\Delta u_n=f&\hbox{in }\Omega\\[.2cm]
(1+\beta h_n)\ds\frac{\de u_n }{\de \nu}+\beta u_n =0&\hbox{on }\partial\Omega.
\end{cases}
\end{align}

Fixed $h_n$, we can indeed consider the auxiliary minimum problem
\begin{align}
 \min_{v\in H^1(\Omega)} F(v,h_n),
\end{align}
whose solution, denoted by $\bar u_n$, exists thanks to the coerciveness condition  \eqref{coerciveness} and solves \eqref{PDEu_n}. Hence, without loss of generality we assume $u_n=\bar u_n$.
Now we observe that $u_n\ge  0$ on $\Omega$. Indeed, assume by contradiction that $\min u_n<0$. By the maximum principle, since $f\ge 0$, $u_n$ achieves its minimum on $\partial\Omega$. Let $\sigma\in\de\Omega$ be a point where the minimum is achieved, then by the second condition in \eqref{PDEu_n} it results
\begin{align*}
 (1+\beta h_n(\sigma))\ds\frac{\de  u_n(\sigma) }{\de \nu}=-\beta  u_n(\sigma)>0,
\end{align*}
which contradicts the maximum principle. 

More precisely, we claim that  there is a constant $K>0$ such that 
\begin{equation}
\label{boundbasso}
u_n(x)\ge U(x)\ge K>0\text{ on }\bar\Omega,
\end{equation}
where $U$ solves
\begin{equation*}\label{PDEU} 
\begin{cases}
-\Delta  U=f&\hbox{in }\Omega\\[.2cm]
\ds\frac{\de  U }{\de \nu}+\beta  U =0& \hbox{on }\partial\Omega.
\end{cases}
\end{equation*}
%{\color{blue}
By the strong maximum principle $U(x)\geq K>0$, so that to prove the claim we have to show that $u_n(x)\geq U(x)$ on $\bar \Omega$. %}
The function $w:=u_{n}-U$ satisfies
\begin{equation*}\label{PDEw}
\begin{cases}
\Delta  w=0&\hbox{in }\Omega\\[.2cm]
 \ds\frac{\de  w }{\de \nu}+\beta  w  \ge 0& \hbox{on }\partial\Omega.
\end{cases}
\end{equation*}
Then, using again the maximum principle we find that $w$ has its minimum on $\de \Omega$, and it cannot be negative. The claim \eqref{boundbasso} follows.

Thanks to \eqref{boundbasso}, we now claim that the sequence $c_{u_{n}}$ cannot vanish as $n\to +\infty$. Here $c_{u_n}$ is the constant in \eqref{c_u} associated to $u_n$. 
To prove this claim, let us observe that by Fatou's Lemma
\[
\liminf_{n\to +\infty} c_{u_{n}} \ge \frac{1}{Per(\Omega)+m\beta} \int_{\de \Omega} \liminf_{n\to +\infty} u_{n} \chi_{\{u_{n} \ge c_{u_{n}}\}} d\sigma,
\]
hence if $c_{u_{n}}\to 0$, being $u_{n}\ge K>0$ we would have $\{u_{n} \ge c_{u_{n}}\}=\de \Omega$ for $n$ sufficiently large and finally
\[
0\ge \frac{K Per(\Omega)}{Per(\Omega)+m\beta}
\]
which is a contradiction.

By Proposition \ref{coer} the sequence $u_{n}$ weakly converges (up to take a subsequence) to a function $u\in H^{1}(\Omega)$. We then observe that 
Proposition \ref{h_exists} allows us to define
\begin{align}
 \bar h_n(\sigma):=\begin{cases}
              \frac{|u_n(\sigma)|}{c_n\beta}-\frac1\beta&\text{if }|u_n(\sigma)|\geq c_n\\
              0&\text{otherwise },
             \end{cases} 
\end{align}
and $(u_n,\bar h_n)$ is still a minimizing sequence. Since $\bar h_{n}$ converges to some function $h\in L^{2}(\de\Omega)\cap \mathcal H_{m}(\partial \Omega)$, it turns out that the couple $(u,h)$ is a minimum of \eqref{minimum_couple}.

The uniqueness follows by the strict convexity of the functional as proved in Proposition \ref{uniqueness_prop}.
\end{proof}

\section{A sharp estimate for the heat content}\label{sec:5}

In this section we deal with Open Problem \ref{op1}. We will denote by $|\Omega|$ the Lebesgue measure of $\Omega$ and by $Per(\Omega)$ its perimeter. 
{
The main result  is the following Theorem.
\begin{theorem}
Let $m>0$ be fixed, $f\equiv 1$, and $(u,h)\in H^1(\Omega)\times L^2(\partial \Omega)$ be the minimizing couple as in Theorem \ref{main_teo}. Then the total heat content satisfies
\begin{equation}\label{isop_estimate}
\int_\Omega u\le \frac{1}{\omega_n^{2/n}n^2}\left(\frac{n|\Omega|^{\frac{2}{n}+1}}{n+2}+|\Omega|^\frac{2}{n} \left(\frac{Per(\Omega)}{\beta}+m\right)\right),
\end{equation}
were $\omega_n$ is the measure of the unit ball in $\R^n$. Equality is achieved when $\Omega$ is a ball and $h$ is constant.
\end{theorem}
\begin{proof}
We denote by $\mu(t)=|\{x\in \Omega: u(x)>t\}|$ the measure of the superlevel set $t$ of the function $u$, and by $P(t)=Per(\{x\in \Omega: u(x)>t\})$ the perimeter of the same set (which exists finite for a.e. $t$). Obviously $\mu(0)=|\Omega|$ and $P(0)=Per(\Omega)$, while $\mu(\max_{\Omega} u)=P(\max_{\Omega} u)=0$.
\newline\noindent \emph{Claim 1. }  The following inequality holds for a.e. $t>0$
\begin{equation}\label{psquare}
P(t)^2\le \mu(t)\left(-\mu'(t)+\int_{\partial \Omega\cap \{u>t\}} \frac{1+\beta h(\sigma)}{\beta u(\sigma)}\,d\mathcal{H}^{n-1}(\sigma)\right).
\end{equation}
%An informal argument which provides the claim
First we use H\"older's inequality, and we get for a.e. $t>0$
\begin{align*}
P(t)^2=& \left(\int_{ \{u=t\}} d\mathcal{H}^{n-1}(\sigma) +\int_{\partial \Omega\cap \partial\{u>t\}} d\mathcal{H}^{n-1}(\sigma) \right)^2\\
\le& \left(\int_{ \{u=t\}} |Du(\sigma)| d\mathcal{H}^{n-1}(\sigma) +\int_{\partial \Omega\cap \partial\{u>t\}} \frac{\beta u(\sigma)}{1+\beta h(\sigma)}\,d\mathcal{H}^{n-1}(\sigma) \right)\\
&\left(\int_{ \{u=t\}} \frac{1}{|Du(\sigma)|} d\mathcal{H}^{n-1}(\sigma)+\int_{\partial \Omega\cap \partial\{u>t\}} \frac{1+\beta h(\sigma)}{\beta u(\sigma)}d\mathcal{H}^{n-1} (\sigma)\right).
\end{align*}
Then by co-area formula, for a.e. $t$, we have $\displaystyle-\mu'(t)\ge \int_{ \{u=t\}} \frac{1}{|Du(\sigma)|} d\mathcal{H}^{n-1}(\sigma)$. 
The Claim is proved if we establish the equality
$$\mu(t)=\int_{ \{u=t\}} |Du(\sigma)| d\mathcal{H}^{n-1}(\sigma) +\int_{\partial \Omega\cap \partial\{u>t\}} \frac{\beta u(\sigma)}{1+\beta h(\sigma)}\,d\mathcal{H}^{n-1}(\sigma).$$
Informally this can be done integrating $\Delta u$ on the superlevel set $t$ of $u$, and then using both the equation and the boundary condition in \eqref{limiting_pb} (here we use that $-\Delta u=f=1$). For a rigorous proof we refer to standard approximation arguments used in \cite{ant,bd,t79} .
\newline\noindent \emph{Claim 2. } For a.e. $t>0$ it holds
\begin{equation}\label{master}
\mu(t)\le\frac{1}{\omega_n^{2/n}n^2}\left(-\mu'(t)\mu(t)^{\frac{2}{n}}+|\Omega|^\frac{2}{n} \int_{\partial \Omega\cap \partial\{u>t\}} \frac{1+\beta h(\sigma)}{\beta u(\sigma)}d\mathcal{H}^{n-1} (\sigma)\right).
\end{equation}
This is a straightforward consequence of \eqref{psquare}, of the isoperimetric inequality
$$\left(\frac{\mu(t)}{\omega_n}\right)^{n-1}\le\left(\frac{P(t)}{n\omega_n}\right)^n,$$
and of the fact that $\mu(t)\le|\Omega|.$

The proof of the Theorem is complete once we integrate both sides of inequality \eqref{master} from $0$ to $\max_{\Omega} u$ and use Fubini's Theorem to deal with the last term (remember that $m=\int_{\partial\Omega}h\, d\mathcal{H}^{n-1}(\sigma)$).

It is easy to check that all inequalities hold as equality if $\Omega$ is a ball and $h$ is constant (notice that in such a case the value of $u$ on $\partial \Omega$ is constant). 
\end{proof} 
\begin{cor}
Solution to Open Problem 1 is the ball, with $h$ constant on the boundary.
\end{cor}
\begin{proof}
Consider a smooth open set $\Omega$. For a given $h$, let $v$ be the solution to \eqref{bubuni}. A solution $u$ to \eqref{limiting_pb} converges to $v$ as $\beta$ goes to $+\infty$. Then inequality \eqref{isop_estimate} becomes
\begin{equation}
\int_\Omega v \;dx\le \frac{1}{\omega_n^{2/n}n^2}\left(\frac{n|\Omega|^{\frac{2}{n}+1}}{n+2}+|\Omega|^\frac{2}{n} m\right). 
\end{equation}
Inequality holds as an equality if $\Omega$ is a ball and $h$ is constant.

%the perimeter on the r.h.s. disappears, and for given $|\Omega|$ and $m$, equality is still achieved when $\Omega$ is a ball and $h$ is constant.
\end{proof}

}
\section*{Acnowledgements}
This work has been partially supported by a MIUR-PRIN 2017 grant ``Qualitative and quantitative aspects of nonlinear PDE's'', PON Ricerca e Innovazione 2014-2020, and by GNAMPA of INdAM.

\end{document}